\documentclass[smallextended,numbook,runningheads]{svjour3}     
\smartqed  

\usepackage{mathptmx}      

\usepackage[T1]{fontenc}      
\usepackage[utf8]{inputenc}   
\usepackage{bbm}              

\usepackage{amsfonts}
\usepackage{graphicx}
\usepackage{xcolor}

\usepackage{amsmath}
\usepackage{amssymb}
\usepackage{commath}          

\usepackage{mathtools}
\usepackage{thmtools}

\usepackage[square,numbers]{natbib}
\usepackage[capitalise]{cleveref}

\usepackage{verbatim}             
\usepackage{etoolbox}             
\usepackage[perpage]{footmisc}    

\allowdisplaybreaks

\DeclareUnicodeCharacter{00A0}{ } 

\DeclarePairedDelimiterX\Set[2]{\lbrace}{\rbrace}%
{ #1 \,:\, #2 } 
\DeclarePairedDelimiterX\inprod[2]{\langle}{\rangle}%
{ #1 , #2 } 

\newcommand{\bigO}{\mathcal{O}} 

\DeclareMathOperator*{\argmin}{arg\,min}  

\DeclarePairedDelimiter\floor{\lfloor}{\rfloor} 

\renewcommand{\epsilon}{\varepsilon}
\renewcommand{\phi}{\varphi}

\newcommand{\R}{\mathbb{R}} 
\newcommand{\N}{\mathbb{N}} 

\DeclareMathOperator{\neper}{e} 

\newcommand{\transpose}{\mathsf{T}}         
\newcommand{\lspan}{\operatorname{span}}    

\newcommand{\rev}[1]{#1}

\journalname{BIT}

\newcommand{\hermite}{\mathrm{H}}
\newcommand{\rkhs}{\mathcal{H}}

\newcommand{\xkgh}{\tilde{x}} 
\newcommand{\wkgh}{\widetilde{w}} 

\newcommand{\textsm}[1]{\text{\tiny{#1}}}

\begin{document}

\title{Gaussian kernel quadrature at scaled Gauss--Hermite nodes\thanks{This work was supported by the Aalto ELEC Doctoral School as well as Academy of Finland projects 266940, 304087, and 313708.
}}


\author{Toni Karvonen         \and
        Simo Särkkä
}


\institute{Department of Electrical Engineering and Automation, Aalto University, Espoo, Finland \\
\email{toni.karvonen@aalto.fi \& simo.sarkka@aalto.fi}
}

\date{}

\maketitle

\begin{abstract}
This article derives an accurate, explicit, and numerically stable approximation to the kernel quadrature weights in one dimension and on tensor product grids when the kernel and integration measure are Gaussian. The approximation is based on use of scaled Gauss--Hermite nodes and truncation of the Mercer eigendecomposition of the Gaussian kernel. Numerical evidence indicates that both the kernel quadrature and the approximate weights at these nodes are positive. An exponential rate of convergence for functions in the reproducing kernel Hilbert space induced by the Gaussian kernel is proved under an assumption on growth of the sum of absolute values of the approximate weights.



\end{abstract}

\section{Introduction}

Let $\mu$ be the standard Gaussian measure on $\R$ and $f \colon \R \to \R$ a measurable function. We consider the problem of numerical computation of the integral with respect to $\mu$ of $f$ using a \emph{kernel quadrature rule} (we reserve the term \emph{cubature} for rules on higher dimensions) based on the Gaussian kernel
\begin{equation}\label{eq:Gaussian}
k(x,y) = \exp\bigg( \! -\frac{(x-y)^2}{2\ell^2} \bigg)
\end{equation}
with the length-scale $\ell > 0$. Given any distinct \emph{nodes} $x_1,\ldots,x_N$, the kernel quadrature rule is an approximation of the form
\begin{equation*}
Q_k(f) \coloneqq \sum_{n=1}^N w_{k,n} f(x_n) \approx \mu(f) \coloneqq \frac{1}{\sqrt{2\pi}} \int_\R f(x) \neper^{-x^2/2} \dif x,
\end{equation*}
with its \emph{weights} $w_k = (w_{k,1},\ldots,w_{k,N}) \in \R^N$ solved from the linear system of equations
\begin{equation}\label{eq:linSys}
K w_k = k_\mu,
\end{equation}
where $[K]_{ij} \coloneqq k(x_i,x_j)$ and $[k_\mu]_i \coloneqq \int_\R k(x_i,x) \dif \mu(x)$. This is equivalent to uniquely selecting the weights such that the $N$ kernel translates $k(x_1,\cdot),\ldots,k(x_N,\cdot)$ are integrated exactly by the quadrature rule. Kernel quadrature rules can be interpreted as best quadrature rules in the reproducing kernel Hilbert space (RKHS) induced by a positive-definite kernel~\citep{Larkin1970}, integrated kernel (radial basis function) interpolants~\sloppy{\citep{Bezhaev1991,SommarivaVianello2006}}, and posteriors to $\mu(f)$ under a Gaussian process prior on the integrand~\sloppy{\citep{Larkin1972,OHagan1991,BriolProbInt2018}}.

\rev{Recently, \citet{FasshauerMcCourt2012} have developed a method to circumvent the well-known problem that interpolation with the Gaussian kernel becomes often numerically unstable---in particular when $\ell$ is large---because the condition number of $K$ tends to grow with an exponential rate~\citep{Schaback1995}. They do this by truncating the Mercer eigendecomposition of the Gaussian kernel after $M$ terms and replacing the interpolation basis $\{k(x_n,\cdot)\}_{n=1}^N$ with the first $M$ eigenfunctions. In this article we show that application of this method with $M=N$ to kernel quadrature yields, when the nodes are selected by a suitable and fairly natural scaling of the nodes of the classical Gauss--Hermite quadrature rule, an accurate, explicit, and numerically stable approximation to the Gaussian kernel quadrature weights. Moreover, the proposed nodes appear to be a good and natural choice for the Gaussian kernel quadrature.}

To be precise, \Cref{thm:main} states that the quadrature rule $\widetilde{Q}_k$ that exactly integrates the first $N$ Mercer eigenfunctions of the Gaussian kernel and uses the nodes
\begin{equation*}
\xkgh_n \coloneqq \frac{1}{\sqrt{2} \alpha \beta} x_n^\textsm{GH}
\end{equation*}
has the weights
\begin{equation*}
\widetilde{w}_{k,n} \coloneqq \bigg(\frac{1}{1+2\delta^2}\bigg)^{1/2} w_n^\textsm{GH} \neper^{\delta^2 \xkgh_n^2} \sum_{m=0}^{\floor{(N-1)/2}} \frac{1}{2^m m! } \bigg( \! \frac{2\alpha^2 \beta^2}{1+2\delta^2} - 1 \bigg)^m \hermite_{2m}(x_n^\textsm{GH}),
\end{equation*}
$\widetilde{w}_k = (\widetilde{w}_{k,1}, \ldots, \widetilde{w}_{k,N}) \in \R^N$, where $\alpha$ (for which the value $1/\sqrt{2}$ seems the most natural), $\beta$, and $\delta$ are constants defined in \Cref{eq:constants}, $\hermite_n$ are the probabilists' Hermite polynomials~\eqref{eq:hermite}, and $x_n^\textsm{GH}$ and $w_n^\textsm{GH}$ are the nodes and weights of the $N$-point Gauss--Hermite quadrature rule. 
We argue that these weights are a good approximation to $w_{k}$ and accordingly call them \emph{approximate Gaussian kernel quadrature weights}. Although we derive no bounds for the error of this weight approximation, numerical experiments in \Cref{sec:numsim} indicate that the approximation is accurate and that it appears that $\widetilde{w}_{k} \to w_{k}$ as $N \to \infty$.
In \Cref{sec:tensor} we extend the weight approximation for $d$-dimensional Gaussian tensor product kernel cubature rules of the form
\begin{equation*}
Q_k^d = Q_{k,1} \otimes \cdots \otimes Q_{k,d},
\end{equation*}
where $Q_{k,i}$ are one-dimensional Gaussian kernel quadrature rules. Since each weight of $Q_k^d$ is a product of weights of the univariate rules, an approximation for the tensor product weights is readily available.

\rev{
It turns out that the approximate weight and the associated nodes $\xkgh_n$ have a number of desirable properties:
\begin{itemize}
\item We are not aware of any work on efficient selection of ``good'' nodes in the setting of this article. The Gauss--Hermite nodes~\citep[Section 3]{OHagan1991} and random points~\citep{RasmussenGhahramani2002} are often used, but one should clearly be able to do better, while computation of the optimal nodes~\cite[Section 5.2]{Oettershagen2017} is computationally demanding. As such, given the desirable properties, listed below, of the resulting kernel quadrature rules, the nodes $\xkgh_n$ appear to be an excellent heuristic choice. These nodes also behave naturally when $\ell \to \infty$; see \Cref{sec:lengthscale}.
\item Numerical experiments in \Cref{sec:numsim-positivity} suggest that both $w_{k,n}$ (for the nodes $\xkgh_n$) and $\widetilde{w}_{k,n}$ are positive for any $N \in \N$ and every $n = 1,\ldots, N$. Besides the optimal nodes, the weights for which are guaranteed to be positive when the Gaussian kernel is used~\citep{RichterDyn1971a,Oettershagen2017}, there are no node configurations that give rise to positive weights as far as we are aware of.
\item Numerical experiments in \Cref{sec:numsim-stability,sec:numsim-positivity} demonstrate that computation of the approximate weights is numerically stable. 
Furthermore, construction of these weights only incurs a quadratic computational cost in the number of points, as opposed to the cubic cost of solving $w_k$ from \Cref{eq:linSys}. See \Cref{sec:complexity} for more details.
Note that to obtain a numerically stable method it is not necessary to use the nodes $\xkgh_n$ as the method in \citep{FasshauerMcCourt2012} can be applied in a straightforward manner for any nodes. However, doing so one forgoes a closed form expression and has to use the QR decomposition.
\item In \Cref{sec:convergence,sec:tensor} we show that slow enough growth with $N$ of $\sum_{i=1}^N \abs[0]{\wkgh_{k,n}}$ (numerical evidence indicates this sum converges to one) guarantees that the approximate Gaussian kernel quadrature rule---as well as the corresponding tensor product version---converges with an exponential rate for functions in the RKHS of the Gaussian kernel. Convergence analysis is based on analysis of magnitude of the remainder of the Mercer expansion and rather explicit bounds on Hermite polynomials and their roots. Magnitude of the nodes $\xkgh_n$ is crucial for the analysis; if they were further spread out the proofs would not work as such.
\item We find the connection to the Gauss--Hermite weights and nodes that the closed form expression for $\widetilde{w}_k$ provides intriguing and hope that it can be at some point used to furnish, for example, a rigorous proof of positivity of the approximate weights.
\end{itemize}
}

\section{Approximate weights}\label{sec:main}

This section contains the main results of this article. The main contribution is derivation, in \Cref{thm:main}, of the weights $\widetilde{w}_k$, that can be used to approximate the kernel quadrature weights. We also discuss positivity of these weights, the effect the kernel length-scale $\ell$ is expected to have on quality of the approximation, and computational complexity.

\subsection{Eigendecomposition of the Gaussian kernel}\label{sec:eigendecomposition}

Let $\nu$ be a probability measure on the real line. If the support of $\nu$ is compact, Mercer's theorem guarantees that any positive-definite kernel $k$ admits an absolutely and uniformly convergent eigendecomposition
\begin{equation}\label{eq:eigenDecomp}
k(x,y) = \sum_{n=0}^\infty \lambda_n \phi_n (x) \phi_n(y)
\end{equation}
for positive and non-increasing eigenvalues $\lambda_n$ and eigenfunctions $\phi_n$ that are included in the RKHS $\rkhs$ induced by $k$ and orthonormal in $L^2(\nu)$. Moreover, $\sqrt{\lambda_n} \phi_n$ are $\rkhs$-orthonormal. If the support of $\nu$ is not compact, the expansion~\eqref{eq:eigenDecomp} converges absolutely and uniformly on all compact subsets of $\R \times \R$ under some mild assumptions~\citep{Sun2005,SteinwartScovel2012}. For the Gaussian kernel~\eqref{eq:Gaussian} and measure the eigenvalues and eigenfunctions are available analytically. For a collection of explicit eigendecompositions of some other kernels, see for instance~\citep[Appendix A]{FasshauerMcCourt2015} 

Let $\mu_\alpha$ stand for the Gaussian probability measure,
\begin{equation*}
\dif \mu_\alpha(x) \coloneqq \frac{\alpha}{\sqrt{\pi}} \neper^{-\alpha^2 x^2} \dif x,
\end{equation*}
with variance $1/(2\alpha^2)$ (i.e., $\mu = \mu_{1/\sqrt{2}}\,$) and
\begin{equation}\label{eq:hermite}
\hermite_n(x) \coloneqq (-1)^n \neper^{x^2/2} \od[n]{}{x} \neper^{-x^2/2}
\end{equation}
for the (unnormalised) probabilists' Hermite polynomial satisfying the orthogonality property $\inprod{\hermite_n}{\hermite_m}_{L^2(\mu)} = n! \, \delta_{nm}$. Denote
\begin{equation}\label{eq:constants}
\epsilon = \frac{1}{\sqrt{2}\ell}, \hspace{0.5cm} \beta = \bigg( 1 + \bigg(\frac{2\epsilon}{\alpha}\bigg)^2 \bigg)^{1/4}, \hspace{0.5cm} \text{and} \hspace{0.5cm} \delta^2 = \frac{\alpha^2}{2} (\beta^2 - 1)
\end{equation}
and note that $\beta > 1$ and $\delta^2 > 0$. Then the eigenvalues and $L^2(\mu_\alpha)$-orthonormal eigenfunctions of the Gaussian kernel are~\citep{FasshauerMcCourt2012}
\begin{equation}\label{eq:eigenvalues}
\lambda_n^\alpha \coloneqq \sqrt{\frac{\alpha^2}{\alpha^2 + \delta^2 + \epsilon^2}} \bigg( \frac{\epsilon^2}{\alpha^2 + \delta^2 + \epsilon^2} \bigg)^{n}
\end{equation}
and
\begin{equation}\label{eq:eigenfunctions}
\phi_n^\alpha(x) \coloneqq \sqrt{\frac{\beta}{n!}} \neper^{-\delta^2 x^2} \hermite_{n} \big( \sqrt{2} \alpha \beta x \big).
\end{equation}
See~\citep[Section 12.2.1]{FasshauerMcCourt2015} for verification that these indeed are Mercer eigenfunctions and eigenvalues for the Gaussian kernel. The role of the parameter $\alpha$ is discussed in \Cref{sec:appweights}. The following result, also derivable from Equation 22.13.17 in~\citep{AbramowitzStegun1964}, will be useful.

\begin{lemma}\label{lemma:integral} The eigenfunctions~\eqref{eq:eigenfunctions} of the Gaussian kernel~\eqref{eq:Gaussian} satisfy
\begin{equation*}
\mu(\phi_{2m+1}^\alpha) = 0 \hspace{0.5cm} \text{and} \hspace{0.5cm} \mu(\phi_{2m}^\alpha) = \bigg(\frac{\beta}{1+2\delta^2}\bigg)^{1/2} \frac{\sqrt{(2m)!}}{2^{m} m! } \bigg(\frac{2\alpha^2 \beta^2}{1+2\delta^2} - 1 \bigg)^m
\end{equation*}
for $m \geq 0$.
\end{lemma}
\begin{proof} Since an Hermite polynomial of odd order is an odd function, $\mu(\phi_{2m+1}^\alpha) = 0$. For even indices, use the explicit expression
\begin{equation*}
\hermite_{2m}(x) = \frac{(2m)!}{2^{m}} \sum_{p=0}^m \frac{(-1)^{m-p}}{(2p)!(m-p)!} \big( \sqrt{2} x \big)^{2p},
\end{equation*}
the Gaussian moment formula
\begin{equation*}
\int_\R x^{2p} \neper^{-\delta^2 x^2} \dif \mu(x) = \frac{1}{\sqrt{2\pi}} \int_\R x^{2p} \neper^{-(\delta^2 + 1/2) x^2} \dif x = \frac{(2p)!}{2^p p!(1+2\delta^2)^{p+1/2}},
\end{equation*}
and the binomial theorem to conclude that
\begin{equation*}
\begin{split}
\mu(\phi_{2m}^\alpha) &= \frac{\sqrt{(2m)!\beta}}{2^{m}} \sum_{p=0}^m \frac{(-1)^{m-p}}{(2p)!(m-p)!} (2 \alpha \beta)^{2p} \int_\R x^{2p} \neper^{-\delta^2 x^2} \dif \mu(x) \\
&= \frac{(-1)^m \sqrt{(2m)! \beta}}{2^{m}\sqrt{1+2\delta^2}} \sum_{p=0}^m \frac{1}{p!(m-p)!} \bigg( -\frac{2\alpha^2 \beta^2}{1+2\delta^2} \bigg)^p \\
&= \frac{(-1)^m \sqrt{(2m)! \beta}}{2^{m} m! \sqrt{1+2\delta^2}} \sum_{p=0}^m {m \choose p} \bigg( -\frac{2\alpha^2 \beta^2}{1+2\delta^2} \bigg)^p \\
&= \bigg(\frac{\beta}{1+2\delta^2}\bigg)^{1/2} \frac{\sqrt{(2m)!}}{2^{m} m! } \bigg(\frac{2\alpha^2 \beta^2}{1+2\delta^2} - 1 \bigg)^m.
\end{split}
\end{equation*} \qed
\end{proof}

\subsection{Approximation via QR decomposition}\label{sec:qr}

We begin by outlining a straightforward extension to kernel quadrature of the work of \citet[Chapter 13]{FasshauerMcCourt2012,FasshauerMcCourt2015} on numerically stable kernel interpolation. Recall that the kernel quadrature weights $w_k \in \R^N$ at distinct nodes $x_1,\ldots,x_N$ are solved from the linear system $K w_k = k_\mu$ with \sloppy{${[K]_{ij} = k(x_i,x_j)}$} and \sloppy{${[k_{\mu}]_i = \int_\R k(x_i,x) \dif \mu(x)}$}. Truncation of the eigendecomposition~\eqref{eq:eigenDecomp} after $M \geq N$ terms\footnote{\rev{Low-rank approximations (i.e.\ $M < N$) are also possible~\citep[Section 6.1]{FasshauerMcCourt2012}.}} yields the approximations $K \approx \Phi \Lambda \Phi^\transpose$ and $k_\mu \approx \Phi \Lambda \phi_\mu$, where \sloppy{${[\Phi]_{ij} \coloneqq \phi_{j-1}^\alpha(x_i)}$} is an $N \times M$ matrix, the diagonal $M \times M$ matrix $[\Lambda]_{ii} \coloneqq \lambda_{i-1}$ contains the eigenvalues in appropriate order, and \sloppy{${[\phi_\mu]_i \coloneqq \mu(\phi_{i-1})}$} is an $M$-vector. The kernel quadrature weights $w_k$ can be therefore approximated by
\begin{equation}\label{eq:Mweights}
\widetilde{w}_k^M \coloneqq \big( \Phi \Lambda \Phi^\transpose \big)^{-1} \Phi \Lambda \phi_\mu.
\end{equation}

\Cref{eq:Mweights} can be written in a more convenient form by exploiting the QR decomposition. The QR decomposition of $\Phi$ is
\begin{equation*}
\Phi = QR \coloneqq Q \begin{bmatrix} R_1 & R_2 \end{bmatrix}
\end{equation*}
for a unitary $Q \in \R^{N \times N}$, an upper triangular $R_1 \in \R^{N \times N}$, and $R_2 \in \R^{N \times (M-N)}$. Consequently, 
\begin{equation*}
\widetilde{w}_k^M = \big( QR \Lambda R^\transpose Q^\transpose \big)^{-1} QR \Lambda \phi_\mu = Q \big( R \Lambda R^\transpose \big)^{-1} R \Lambda \phi_\mu.
\end{equation*}
The decomposition
\begin{equation*}
\Lambda = \begin{bmatrix} \Lambda_1 & 0 \\ 0 & \Lambda_2 \end{bmatrix}
\end{equation*}
of $\Lambda \in \R^{M \times M}$ into diagonal $\Lambda_1 \in \R^{N \times N}$ and $\Lambda_2 \in \R^{(M-N) \times (M-N)}$ allows for writing
\begin{equation*}
R \Lambda R^\transpose = R_1 \Lambda_1 \big( R_1^\transpose + \Lambda_1^{-1} R_1^{-1} R_2 \Lambda_2 R_2^\transpose \big).
\end{equation*}
Therefore,
\begin{equation}\label{eq:MweightsFinal}
\widetilde{w}_k^M = Q \big( R_1^\transpose + \Lambda_1^{-1} R_1^{-1} R_2 \Lambda_2 R_2^\transpose \big)^{-1} \begin{bmatrix} I_N & \Lambda_1^{-1} R_1^{-1} R_2 \Lambda_2 \end{bmatrix} \phi_\mu,
\end{equation}
where $I_N$ is the $N \times N$ identity matrix. \rev{If $\epsilon^2/(\alpha^2+\delta^2+\epsilon^2)$ is small (i.e., $\ell$ is large), numerical ill-conditioning in \Cref{eq:MweightsFinal} for the Gaussian kernel is associated with the diagonal matrices $\Lambda_1^{-1}$ and $\Lambda_2$.}
Consequently, numerical stability can be significantly improved by performing the multiplications by these matrices in the terms $\Lambda_1^{-1} R_1^{-1} R_2 \Lambda_2 R_2^\transpose$ and $\Lambda_1^{-1} R_1^{-1} R_2 \Lambda_2$ analytically; \rev{see \citep[Sections 4.1 and 4.2]{FasshauerMcCourt2012} for more details}. 

Unfortunately, using the QR decomposition does not provide an attractive closed form solution for the approximate weights $\widetilde{w}_k^M$ \rev{for general $M$}. Setting $M = N$ turns $\Phi$ into a square matrix, enabling its direct inversion and formation of an \rev{explicit} connection to the classical Gauss--Hermite quadrature. \rev{The rest of the article is concerned with this special case.}

\subsection{Gauss--Hermite quadrature}

Given a measure $\nu$ on $\R$, the $N$-point \emph{Gaussian quadrature rule} is the unique $N$-point quadrature rule that is exact for all polynomials of degree at most $2N-1$. We are interested in \emph{Gauss--Hermite quadrature rules} that are Gaussian rules for the Gaussian measure~$\mu$:
\begin{equation*}
\sum_{n=1}^N w_n^\textsm{GH} p(x_n^\textsm{GH}) = \mu(p)
\end{equation*}
for every polynomial $p \colon \R \to \R$ with $\deg p \leq 2N-1$. The nodes $x_1^\textsm{GH}, \ldots x_N^\textsm{GH}$ are the roots of the $N$th Hermite polynomial $\hermite_N$ and the weights $w_1^\textsm{GH}, \ldots, w_N^\textsm{GH}$ are positive and sum to one. The nodes and the weights are related to the eigenvalues and eigenvectors of the tridiagonal Jacobi matrix formed out of three-term recurrence relation coefficients of normalised Hermite polynomials~\citep[Theorem 3.1]{Gautschi2004}.

We make use of the following theorem, a one-dimensional special case of a more general result due to \citet{Mysovskikh1968}. See also~\citep[Section 7]{Cools1997}. This result also follows from the Christoffel--Darboux formula~\eqref{eq:christoffelDarboux}.

\begin{theorem}\label{thm:Mysovskikh} Let $\nu$ be a measure on $\R$. Suppose that $x_1,\ldots,x_N$ and $w_1,\ldots,w_N$ are the nodes and weights of the unique Gaussian quadrature rule. Let $p_0,\ldots,p_{N-1}$ be the $L^2(\nu)$-orthonormal polynomials. Then the matrix $[P]_{ij} \coloneqq \sum_{n=0}^{N-1} p_n(x_i) p_n(x_j)$ is diagonal and has the diagonal elements $[P]_{ii} = 1/w_i$.
\end{theorem}

\subsection{Approximate weights at scaled Gauss--Hermite nodes}\label{sec:appweights}

Let us now consider the approximate weights~\eqref{eq:Mweights} with $M = N$. Assuming that $\Phi$ is invertible, we then have
\begin{equation*}
w_k \approx \widetilde{w}_k \coloneqq \widetilde{w}_k^N = \big( \Phi \Lambda \Phi^\transpose \big)^{-1} \Phi \Lambda \phi_\mu = \Phi^{-\transpose} \phi_\mu.
\end{equation*}
\rev{Note that the exponentially decaying Mercer eigenvalues, a major source of numerical instability, do not appear in the equation for $\widetilde{w}_k$.} The weights $\widetilde{w}_k$ are those of the unique quadrature rule that is exact for the $N$ first eigenfunctions $\phi_0^\alpha,\ldots,\phi_{N-1}^\alpha$.
For the Gaussian kernel, we are in a position to do much more. Recalling the form of the eigenfunctions in \Cref{eq:eigenfunctions}, we can write $\Phi = \sqrt{\beta} E^{-1} V$ for the diagonal matrix $[E]_{ii} \coloneqq \neper^{\delta^2 x_i^2}$ and the Vandermonde matrix 
\begin{equation}\label{eq:Vmatrix}
[V]_{ij} \coloneqq \frac{1}{\sqrt{(j-1)!}} \hermite_{j-1} \big( \sqrt{2}\alpha\beta x_i \big)
\end{equation}
of scaled and normalised Hermite polynomials. From this it is evident that $\Phi$ is invertible---which is just a manifestation of the fact that the eigenfunctions of a totally positive kernel constitute a Chebyshev system~\citep{Kellog1918,Pinkus1996}. Consequently,
\begin{equation*}
\widetilde{w}_k = \frac{1}{\sqrt{\beta}} E V^{-\transpose} \phi_\mu.
\end{equation*}
Select the nodes
\begin{equation*}
\xkgh_n \coloneqq \frac{1}{\sqrt{2} \alpha \beta} x_n^\textsm{GH}.
\end{equation*}
Then the matrix $V$ defined in \Cref{eq:Vmatrix} is precisely the Vandermonde matrix of the normalised Hermite polynomials and $V V^\transpose$ is the matrix $P$ of \Cref{thm:Mysovskikh}. Let $W_\textsm{GH}$ be the diagonal matrix containing the Gauss--Hermite weights. It follows that $V^{-\transpose} = W_\textsm{GH} V$ and
\begin{equation}\label{eq:wa1st}
\widetilde{w}_k = \frac{1}{\sqrt{\beta}} E V^{-\transpose} \phi_\mu = \frac{1}{\sqrt{\beta}} E W_\textsm{GH} V \phi_\mu.
\end{equation}
Combining this equation with \Cref{lemma:integral}, we obtain the main result of this article.

\begin{theorem}\label{thm:main} Let $x_1^\textsm{GH}, \ldots, x_N^\textsm{GH}$ and $w_1^\textsm{GH}, \ldots, w_N^\textsm{GH}$ stand for the nodes and weights of the $N$-point Gauss--Hermite quadrature rule. Define the nodes
\begin{equation}\label{eq:nodes}
\xkgh_n = \frac{1}{\sqrt{2} \alpha \beta} x_n^\textsm{GH}.
\end{equation}
Then the weights $ \wkgh_{k} \in \R^N$ of the $N$-point quadrature rule
\begin{equation*}
\widetilde{Q}_k(f) \coloneqq \sum_{n=1}^N \wkgh_{k,n} f(\xkgh_n),
\end{equation*}
defined by the exactness conditions $\widetilde{Q}_k(\phi_n^\alpha) = \mu_\alpha(\phi_n^\alpha)$ for $n = 0,\ldots,N-1$, are
\begin{equation}\label{eq:approximation}
\widetilde{w}_{k,n} = \bigg(\frac{1}{1+2\delta^2}\bigg)^{1/2} w_n^\textsm{GH} \neper^{\delta^2 \xkgh_n^2} \sum_{m=0}^{\floor{(N-1)/2}} \frac{1}{2^m m! } \bigg( \! \frac{2\alpha^2 \beta^2}{1+2\delta^2} - 1 \bigg)^m \hermite_{2m}(x_n^\textsm{GH}),
\end{equation}
where $\alpha$, $\beta$, and $\delta$ are defined in \Cref{eq:constants} and $\hermite_{2m}$ are the probabilists' Hermite polynomials~\eqref{eq:hermite}. 
\end{theorem}

Since the weights $\wkgh_{k}$ are obtained by truncating of the Mercer expansion of $k$, it is to be expected that $\wkgh_k \approx w_k$. This motivates our calling of these weights the \emph{approximate Gaussian kernel quadrature weights}. We do not provide theoretical results on quality of this approximation, but the numerical experiments in \Cref{sec:numsim-approximation} indicate that the approximation is accurate and that its accuracy increases with $N$. See~\citep{FasshauerMcCourt2012} for related experiments.

An alternative non-analytical formula for the approximate weights can be derived using the Christoffel--Darboux formula~\citep[Section 1.3.3]{Gautschi2004}
\begin{equation}\label{eq:christoffelDarboux}
\sum_{m=0}^M \frac{\hermite_m(x)\hermite_m(y)}{m!} = \frac{\hermite_M(y) \hermite_{M+1}(x) - \hermite_M(x) \hermite_{M+1}(y)}{M! (x-y)}.
\end{equation}
From~\Cref{eq:wa1st} we then obtain (keep in mind that $x_1^\textsm{GH},\ldots,x_N^\textsm{GH}$ are the roots of~$\hermite_N$)
\begin{equation*}
\begin{split}
\widetilde{w}_{k,n} &= \frac{1}{\sqrt{\beta}} w_n^\textsm{GH} \neper^{\delta^2 \xkgh_n^2} \sum_{m=0}^{N-1} \frac{1}{\sqrt{m!}} \hermite_m(x_n^\textsm{GH}) \mu(\phi_m^\alpha) \\
&= w_n^\textsm{GH} \neper^{\delta^2 \xkgh_n^2} \int_\R \neper^{-\delta^2 x^2} \sum_{m=0}^{N-1} \frac{\hermite_m(x_n^\textsm{GH}) \hermite_m(\sqrt{2}\alpha \beta x)}{m!} \dif \mu(x) \\
&= \frac{w_n^\textsm{GH} \neper^{\delta^2 \xkgh_n^2} \hermite_{N-1}(x_n^\textsm{GH})}{\sqrt{2\pi} (N-1)!} \int_\R \frac{\hermite_N(\sqrt{2}\alpha \beta x) }{\sqrt{2}\alpha \beta x-x_n^\textsm{GH}} \neper^{-(\delta^2 + 1/2) x^2} \dif x \\
&= \frac{w_n^\textsm{GH} \neper^{\delta^2 \xkgh_n^2} \hermite_{N-1}(x_n^\textsm{GH})}{2\sqrt{\pi}\alpha\beta(N-1)!} \int_\R \frac{\hermite_N(x) }{x-x_n^\textsm{GH}} \exp\bigg( -\frac{\delta^2 + 1/2}{2\alpha^2\beta^2} x^2 \bigg) \dif x.
\end{split}
\end{equation*}
This formula is analogous to the formula
\begin{equation*}
w_n^\textsm{GH} = \frac{1}{\sqrt{2\pi} N \hermite_{N-1}(x_n^\textsm{GH})} \int_\R \frac{\hermite_N(x)}{x-x_n^\textsm{GH}} \neper^{-x^2/2} \dif x
\end{equation*}
for the Gauss--Hermite weights. Plugging this in, we get
\begin{equation*}
\widetilde{w}_{k,n} = \frac{\neper^{\delta^2 \xkgh_n^2}}{2\sqrt{2} \pi \alpha \beta N!} \int_\R \frac{\hermite_N(x)}{x-x_n^\textsm{GH}} \neper^{-x^2/2} \dif x \int_\R \frac{\hermite_N(x) }{x-x_n^\textsm{GH}} \exp\bigg( -\frac{\delta^2 + 1/2}{2\alpha^2\beta^2} x^2 \bigg) \dif x.
\end{equation*}

It appears that both $w_{k,n}$ and $\widetilde{w}_{k,n}$ of \Cref{thm:main} are positive for many choices of~$\alpha$; see \Cref{sec:numsim-positivity} for experiments involving $\alpha = 1/\sqrt{2}$. Unfortunately, we have not been able to prove this. In fact, numerical evidence indicates something slightly stronger. Namely that the even polynomial
\begin{equation*}
R_{\gamma,N}(x) \coloneqq \sum_{m=0}^{\floor{(N-1)/2}} \frac{\gamma^m}{2^m m! } \hermite_{2m}(x)
\end{equation*}
of degree $2\floor{(N-1)/2}$ is positive for every $N \geq 1$ and (at least) every $0 < \gamma \leq 1$. This would imply positivity of $\widetilde{w}_{k,n}$ since the Gauss--Hermite weights $w_n^\textsm{GH}$ are positive. For example, with $\alpha = 1/\sqrt{2}$,
\begin{equation*}
\frac{2\alpha^2 \beta^2}{1+2\delta^2} - 1 = \frac{2\sqrt{1+8\epsilon^2}}{1+\sqrt{1+8\epsilon^2}}
- 1 = \frac{\sqrt{1+8\epsilon^2}-1}{1+\sqrt{1+8\epsilon^2}} \in (0, 1).
\end{equation*}

As discussed in~\citep{FasshauerMcCourt2012} in the context of kernel interpolation, the parameter $\alpha$ acts as a global scale parameter. While in interpolation it is not entirely clear how this parameter should be selected, in quadrature it seems natural to set $\alpha = 1/\sqrt{2}$ so that the eigenfunctions are orthonormal in $L^2(\mu)$. This is the value that we use, though also other values are potentially of interest since $\alpha$ can be used to control the spread of the nodes independently of the length-scale $\ell$. In \Cref{sec:convergence}, we also see that this value leads to more natural convergence analysis.

\subsection{Effect of the length-scale}\label{sec:lengthscale}

Roughly speaking, magnitude of the eigenvalues
\begin{equation*}
\lambda_n^\alpha = \sqrt{\frac{\alpha^2}{\alpha^2 + \delta^2 + \epsilon^2}} \bigg( \frac{\epsilon^2}{\alpha^2 + \delta^2 + \epsilon^2} \bigg)^{n}
\end{equation*}
determines how many eigenfunctions are necessary for an accurate weight approximation. We therefore expect that the approximation~\eqref{eq:approximation} is less accurate when the length-scale $\ell$ is small (i.e., $\epsilon = 1/(\sqrt{2}\ell)$ is large). This is confirmed by the numerical experiments in \Cref{sec:numsim}. 

Consider then the case $\ell \to \infty$. This scenario is called the \emph{flat limit} in scattered data approximation literature where it has been proved\footnote{It is interesting to note that the first published observation of analogous phenomenon is, as far as we are aware of, due to~\citet[Section 3.3]{OHagan1991} in kernel quadrature literature, predating the work of \citet{DriscollFornberg2002}. See also~\citep{Minka2000} for early quadrature-related work on the topic.} that the kernel interpolant associated to an isotropic kernel with increasing length-scale converges to (i) the unique polynomial interpolant of degree $N-1$ to the data if the kernel is infinitely smooth~\citep{LarssonFornberg2005,Schaback2005,LeeYoonYoon2007} or (ii) to a polyharmonic spline interpolant if the kernel is of finite smoothness~\citep{LeeMicchelliYoon2014}. In our case, $\ell \to \infty$ results in
\begin{equation*}
\epsilon \to 0, \hspace{0.5cm} \beta \to 1, \hspace{0.5cm} \delta^2 \to 0, \hspace{0.5cm} \lambda_n^\alpha \to 0, \hspace{0.5cm} \text{and} \hspace{0.5cm} \phi_n^\alpha(x) \to \hermite_n \big(\sqrt{2} \alpha x \big).
\end{equation*}
If the nodes are selected as in \Cref{eq:nodes}, $\xkgh_n \to x_n^\textsm{GH}/(\sqrt{2}\alpha)$. That is, if $\alpha = 1/\sqrt{2}$
\begin{equation*}
\phi_n^\alpha(x) \to \hermite_n(x), \hspace{0.5cm} \xkgh_n \to x_n^\textsm{GH}, \hspace{0.5cm} \text{and} \hspace{0.5cm} \widetilde{w}_{k,n} \to w_n^\textsm{GH}.
\end{equation*}
That the approximate weights convergence to the Gauss--Hermite ones can be seen, for example, from \Cref{eq:approximation} by noting that only the first term in the sum is retained at the limit. Based on the aforementioned results regarding convergence of kernel interpolants to polynomial ones at the flat limit, it is to be expected that also $w_{k,n} \to w_n^\textsm{GH}$ as $\ell \to \infty$ (we do not attempt to prove this). Because the Gauss--Hermite quadrature rule is the ``best'' for polynomials and kernel interpolants convergence to polynomials at the flat limit, the above observation provides another justification for the choice $\alpha = 1/\sqrt{2}$ that we proposed the preceding section.

When it comes to node placement, the length-scale is having an intuitive effect if the nodes are selected according to \Cref{eq:nodes}. For small $\ell$, the nodes are placed closer to the origin where most of the measure is concentrated as integrands are expected to converge quickly to zero as $\abs[0]{x} \to \infty$, whereas for larger $\ell$ the nodes are more---but not unlimitedly---spread out in order to capture behaviour of functions that potentially contribute to the integral also further away from the origin.

\subsection{On computational complexity} \label{sec:complexity}

Because the Gauss--Hermite nodes and weights are related to the eigenvalues and eigenvectors of the tridiagonal Jacobi matrix~\citep[Theorem 3.1]{Gautschi2004} they\rev{---and the points $\tilde{x}_n$---}can be solved in quadratic time \rev{(in practice, these nodes and weights can be often tabulated beforehand)}.
From \Cref{eq:approximation} it is seen that computation of each approximate weight is linear in $N$: there are approximately $(N-1)/2$ terms in the sum and the Hermite polynomials can be evaluated on the fly using the three-term recurrence formula \rev{$\hermite_{n+1}(x) = x \hermite_n(x) - n \hermite_{n-1}(x)$}. That is, computational cost of obtaining $\tilde{x}_n$ and $\widetilde{w}_{k,n}$ for $n=1,\ldots,N$ is \rev{quadratic} in $N$. Since the kernel matrix $K$ of the Gaussian kernel is dense, solving the exact kernel quadrature weights from the linear system~\eqref{eq:linSys} \rev{for the points $\tilde{x}_n$} incurs a \rev{more demanding} cubic computational cost.
Because computational cost of a tensor product rule does not depend on the nodes and weights after these have been computed, the above discussion also applies to the rules presented in \Cref{sec:tensor}.

\section{Convergence analysis}\label{sec:convergence}

In this section we analyse convergence in the reproducing kernel Hilbert space $\rkhs \subset C^\infty(\R)$ induced by the Gaussian kernel of quadrature rules that are exact for the Mercer eigenfunctions. First, we prove a generic result (\Cref{thm:convergence}) to this effect and then apply this to the quadrature rule with the nodes $\xkgh_n$ and weights $\widetilde{w}_{k,n}$. If $\sum_{n=1}^N \abs[0]{\widetilde{w}_{k,n}}$ does not grow too fast with $N$, we obtain exponential convergence rates. 

Recall some basic facts about reproducing kernel Hilbert spaces spaces~\citep{BerlinetThomasAgnan2004}: (i) $\inprod{f}{k(x,\cdot)}_\rkhs = f(x)$ for any $f \in \rkhs$ and $x \in \R$ and (ii) $f = \sum_{n=0}^\infty \lambda_n^\alpha \inprod{f}{\phi_n^\alpha} \phi_n^\alpha$ for any $f \in \rkhs$. The \emph{worst-case error} $e(Q)$ of a quadrature rule $Q(f) = \sum_{n=1}^N w_n f(x_n)$ is 
\begin{equation*}
e(Q) \coloneqq \sup_{\norm[0]{f}_\rkhs \leq 1} \abs[0]{\mu(f) - Q(f)}.
\end{equation*}
Crucially, the worst-case error satisfies
\begin{equation*}
\abs[0]{\mu(f) - Q(f)} \leq \norm[0]{f}_\rkhs e(Q)
\end{equation*}
for any $f \in \rkhs$. This justifies calling a sequence $\{Q_N\}_{N=1}^\infty$ of $N$-point quadrature rules \emph{convergent} if $e(Q_N) \to 0$ as $N \to \infty$. For given nodes $x_1,\ldots,x_N$, the weights $w_k = (w_{k,1},\ldots,w_{k,N})$ of the kernel quadrature rule $Q_k$ are unique minimisers of the worst-case error:
\begin{equation*}
w_k = \argmin_{w \in \R^N} \sup_{\norm[0]{f}_\rkhs \leq 1} \, \abs[3]{ \int_\R f \dif \mu - \sum_{n=1}^N w_i f(x_i) }.
\end{equation*}
It follows that a rate of convergence to zero for $e(Q)$ also applies to $e(Q_k)$.

A number of convergence results for kernel quadrature rules on compact spaces appear in~\citep{Bezhaev1991,KanagawaSriperumbudurFukumizu2017,BriolProbInt2018}. When it comes to the RKHS of the Gaussian kernel, characterised in~\citep{Steinwart2006,Minh2010}, \citet{KuoWozniakowski2012} have analysed convergence of the Gauss--Hermite quadrature rule. Unfortunately, it turns out that the Gauss--Hermite rule converges in this space if and only if $\epsilon^2 < 1/2$. Consequently, we believe that the analysis below is the first to establish convergence, under the assumption (supported by our numerical experiments) that the sum of $\abs[0]{\widetilde{w}_{k,n}}$ does not grow too fast, of an explicitly constructed sequence of quadrature rules in the RKHS of the Gaussian kernel with any value of the length-scale parameter. We begin with two simple lemmas.

\begin{lemma}\label{lemma:eigFuncBound} The eigenfunctions $\phi_n^\alpha$ admit the bound
\begin{equation*}
\sup_{n \geq 0} \, \abs[0]{ \phi_n^\alpha(x) } \leq K \sqrt{\beta} \neper^{\alpha^2 x^2/2}
\end{equation*}
for a constant $K \leq 1.087$ and every $x \in \R$.
\end{lemma}
\begin{proof} For each $n \geq 0$, the Hermite polynomials obey the bound
\begin{equation}\label{eq:HermitePolyBound}
\frac{1}{n!} \, H_n(x)^2 \leq K^2 \neper^{x^2/2}
\end{equation}
for a constant $K \leq 1.087$~\citep[p.\ 208]{Erdelyi1953}. See~\citep{BonanClark1990} for other such bounds\footnote{In particular, the factor $n^{-1/6}$ could be added on the right-hand side. This would make little difference in convergence analysis of \Cref{thm:convergence}.}. Thus
\begin{equation*}
\phi_n^\alpha(x)^2 = \frac{\beta}{n!} \neper^{-2\delta^2 x^2} \hermite_n \big( \sqrt{2}\alpha\beta x \big)^2 \leq K^2 \beta \exp\big( (\alpha^2\beta^2 - 2\delta^2) x^2 \big) = K^2 \beta \neper^{\alpha^2 x^2}.
\end{equation*} \qed
\end{proof}

\begin{lemma}\label{lemma:ConvConstant} Let $\alpha = 1/\sqrt{2}$. Then
\begin{equation*}
\sqrt{ \frac{\epsilon^2}{1/2 + \delta^2 + \epsilon^2} } \neper^{\rho/(2\beta^2)} \in (0, 1)
\end{equation*}
for every $\ell > 0$ if and only if $\rho \leq 2$.
\end{lemma}
\begin{proof}
The function 
\begin{equation*}
\gamma(\epsilon^2) \coloneqq \frac{\epsilon^2}{1/2 + \delta^2 + \epsilon^2} \neper^{\rho/\beta^2}
\end{equation*}
satisfies $\gamma(0) = 0$ and $\gamma(\epsilon^2) \to 1$ as $\epsilon^2 \to \infty$. The derivative 
\begin{equation*}
\frac{\dif \gamma(\epsilon^2)}{\dif \epsilon^2} = \frac{4\neper^{\rho/\beta^2} (1 + 4(2-\rho)\epsilon^2) }{(4\epsilon^2 + \beta^2 +1)\beta^3}
\end{equation*}
is positive when $\rho \leq 2$. For $\rho > 2$, the derivative has a single root at $\epsilon_0^2 = 1/(4(\rho-2))$ so that $\gamma(\epsilon_0^2) > 1$. That is, $\gamma(\epsilon^2) \in (0,1)$, and consequently $\gamma(\epsilon^2)^{1/2} \in (0,1)$, if and only if $\rho \leq 2$. \qed
\end{proof}

\begin{theorem}\label{thm:convergence} Let $\alpha = 1/\sqrt{2}$. Suppose that the nodes $x_1,\ldots,x_N$ and weights $w_1,\ldots,w_N$ of an $N$-point quadrature rule $Q_N$ satisfy
\begin{enumerate}
\item $\sum_{n=1}^N \abs[0]{w_{n}} \leq W_N$ for some $W_N \geq 0$;
\item $Q_N(\phi_n^\alpha) = \mu(\phi_n^\alpha)$ for each $n = 0,\ldots,M_N-1$ for some $M_N \geq 1$;
\item $\sup_{1\leq n \leq N}\abs[0]{x_{n}} \leq 2 \sqrt{M_N} / \beta$.
\end{enumerate}
Then there exist constants $C_1,C_2 > 0$, independent of $N$ and $Q_N$, and $0 < \eta < 1$ such that
\begin{equation*}
e(Q_N) \leq (1 + C_1 W_N) C_2 \eta^{M_N}.
\end{equation*}
Explicit forms of these constants appear in \Cref{eq:ConvConstants}.
\end{theorem}

\begin{proof}
For notational convenience, denote
\begin{equation*}
\lambda_n^\alpha = \lambda_n = \sqrt{\frac{1/2}{1/2 + \delta^2 + \epsilon^2}} \bigg( \frac{\epsilon^2}{1/2 + \delta^2 + \epsilon^2} \bigg)^{n} = \tau \lambda^n
\end{equation*}
and $\phi_n = \phi_n^\alpha$. Because every $f \in \rkhs$ admits the expansion $f = \sum_{n=0}^\infty \lambda_n \inprod{f}{\phi_n}_\rkhs \phi_n$ and $Q_N(\phi_n) = \mu(\phi_n)$ for $n < M_N$, it follows from the Cauchy--Schwarz inequality and $\norm[0]{\phi_n}_\rkhs = 1/\sqrt{\lambda_n}$ that
\begin{equation}\label{eq:truncation-bound}
\begin{split}
\abs[0]{\mu(f) - Q_N(f)} &= \abs[3]{\sum_{n=M_N}^\infty \lambda_n \inprod{f}{\phi_n}_\rkhs \, [ \mu(\phi_n) - Q_N(\phi_n) ]} \\
&\leq \norm[0]{f}_\rkhs \sum_{n=M_N}^\infty \lambda_n^{1/2} \abs[0]{ \mu(\phi_n) - Q_N(\phi_n) }.
\end{split}
\end{equation}
From \Cref{lemma:eigFuncBound} we have $\abs[0]{\phi_n(x)} \leq K \sqrt{\beta} \neper^{x^2/4}$ for a constant $K \leq 1.087$. Consequently, the assumption $\sup_{1\leq m \leq N} \abs[0]{x_m} \leq 2 \sqrt{M_N}/\beta$ yields
\begin{equation*}
\sup_{1\leq m \leq N} \, \sup_{n \geq 0} \, \abs[0]{\phi_n(x_m)} \leq K \sqrt{\beta} \neper^{M_N/\beta^2}.
\end{equation*}
Combining this with Hölder's inequality and $L^2(\mu)$-orthonormality of $\phi_n$, that imply $\mu(\phi_n) \leq \mu(\phi_n^2)^{1/2} = 1$, we obtain the bound
\begin{equation}\label{eq:PhiError}
\abs[0]{ \mu(\phi_n) - Q_N(\phi_n) } \leq 1 + \sum_{m=1}^N \abs[0]{w_m} \abs[0]{\phi_n(x_m)} \leq 1 + K \sqrt{\beta} W_N \neper^{M_N/\beta^2}.
\end{equation}
Inserting this into \Cref{eq:truncation-bound} produces
\begin{equation}\label{eq:ConvConstants}
\begin{split}
\abs[0]{\mu(f) - Q_N(f)} &\leq \norm[0]{f}_\rkhs \big( 1 + W_N K \sqrt{\beta} \neper^{M_N/\beta^2} \big) \sum_{n=M_N}^\infty \lambda_n^{1/2} \\
&= \norm[0]{f}_\rkhs \big( 1 + K \sqrt{\beta} W_N \neper^{M_N/\beta^2} \big) \sqrt{\tau} \sum_{n=M_N}^\infty \lambda^{n/2} \\
&= \norm[0]{f}_\rkhs \big( 1 + K \sqrt{\beta} W_N \neper^{M_N/\beta^2} \big) \frac{\sqrt{\tau}}{1-\sqrt{\lambda}} \lambda^{M_N/2} \\ 
&\leq \norm[0]{f}_\rkhs \big( 1 + K \sqrt{\beta} W_N \big) \frac{\sqrt{\tau}}{1-\sqrt{\lambda}} \big( \sqrt{\lambda} \neper^{1/\beta^2} \big)^{M_N}.
\end{split}
\end{equation}
Noticing that $\sqrt{\lambda} \neper^{1/\beta^2} < 1$ by \Cref{lemma:ConvConstant} concludes the proof. \qed
\end{proof}

\begin{remark} From \Cref{lemma:ConvConstant} we observe that the proof does not yield $\eta < 1$ (for every~$\ell$) if the assumption $\sup_{1\leq n \leq N}\abs[0]{x_{n}} \leq 2 \sqrt{M_N} / \beta$ on placement of the nodes is relaxed by replacing the constant $2$ on the right-hand side with $C > 2$.
\end{remark}

Consider now the $N$-point approximate Gaussian kernel quadrature rule \sloppy{${\widetilde{Q}_{k,N} = \sum_{n=1}^N \widetilde{w}_{k,n} f(\xkgh_n)}$} whose nodes and weights are defined in \Cref{thm:main} and set $\alpha = 1/\sqrt{2}$. The nodes $x_n^\textsm{GH}$ of the $N$-point Gauss--Hermite rule admit the bound~\citep{AreaDimitrovGodoyRonveaux2004}
\begin{equation*}
\sup_{1 \leq n \leq N} \abs[0]{x_n^\textsm{GH}} \leq 2\sqrt{N-1}
\end{equation*}
for every $N \geq 1$. That is,
\begin{equation*}
\xkgh_n = \frac{1}{\beta} x_n^\textsm{GH} \leq \frac{2\sqrt{N}}{\beta}.
\end{equation*}
Since the rule $\widetilde{Q}_{k,N}$ is exact for the first $N$ eigenfunctions, $M_N = N$. Hence the assumption on placement of the nodes in \Cref{thm:convergence} holds. As our numerical experiments indicate that the weights $\widetilde{w}_{k,n}$ are positive and $\sum_{n=1}^N \abs[0]{\widetilde{w}_{k,n}} \to 1$ as $N \to \infty$, it seems that the exponential convergence rate of \Cref{thm:convergence} is valid for $\widetilde{Q}_{k,N}$ (as well as for the corresponding kernel quadrature rule $Q_{k,N}$) with $M_N = N$. Naturally, this result is valid whenever the growth of the absolute weight sum is, for example, polynomial in~$N$.

\begin{theorem}\label{thm:ConvergenceSpecific} Let $\alpha = 1/\sqrt{2}$ and suppose that $\sup_{N \geq 1} \sum_{n=1}^N \abs[0]{\widetilde{w}_{k,n}} < \infty$. Then the quadrature rules $\widetilde{Q}_{k,N}(f) = \sum_{n=1}^N \widetilde{w}_{k,n} f(\xkgh_n)$ and $Q_{k,N}(f) = \sum_{n=1}^N w_{k,n} f(\xkgh_n)$ satisfy
\begin{equation*}
e(Q_{k,N}) \leq e(\widetilde{Q}_{k,N}) = \bigO(\eta^N)
\end{equation*}
for $0 < \eta < 1$.
\end{theorem}

Another interesting case are the \emph{generalised Gaussian quadrature rules}\footnote{Note that the cited results are for kernels and functions on compact intervals. However, generalisations for the whole real line are possible~\citep[Chapter VI]{KarlinStudden1966}.} for the eigenfunctions. As the eigenfunctions constitute a complete Chebyshev system~\mbox{\citep{Kellog1918,Pinkus1996}}, there exists a quadrature rule $Q^*_N$ with positive weights $w_1^*,\ldots,w_N^*$ such that $Q_N^*(\phi_n) = \mu(\phi_n)$ for every $n=0,\ldots,2N-1$~\citep{Barrow1978}. Appropriate control of the nodes of these quadrature rules would establish an exponential convergence result with the ``double rate'' $M_N = 2N$.

\section{Tensor product rules}\label{sec:tensor}

Let $Q_1,\ldots,Q_d$ be quadrature rules on $\R$ with nodes $X_i = \{ x_{i,1},\ldots,x_{i,N_i} \}$ and weights $w_{1}^i,\ldots,w_{N_i}^i$ for each $i=1,\ldots,d$. The \emph{tensor product rule} on the Cartesian grid \sloppy{${X \coloneqq X_1 \times \cdots \times X_d \subset \R^d}$} is the cubature rule
\begin{equation}\label{eq:tensorRule}
Q^d(f) \coloneqq (Q_1 \otimes \cdots \otimes Q_d)(f) = \sum_{\mathcal{I} \leq \mathcal{N}} w_\mathcal{I} f(x_\mathcal{I}),
\end{equation}
where $\mathcal{I} \in \N^d$ is a multi-index, $\mathcal{N} \coloneqq (N_1,\ldots,N_d) \in \N^d$, and the nodes and weights are
\begin{equation*}
x_\mathcal{I} \coloneqq (x_{1,\mathcal{I}(1)}, \ldots x_{d,\mathcal{I}(d)}) \in X \hspace{0.5cm} \text{and} \hspace{0.5cm} w_{\mathcal{I}} \coloneqq \prod_{i=1}^d w_{\mathcal{I}(i)}^i.
\end{equation*}
We equip $\R^d$ with the $d$-variate standard Gaussian measure
\begin{equation}\label{eq:measureMultiD}
\dif \mu^d(x) \coloneqq (2\pi)^{-d/2} \neper^{-\norm[0]{x}^2/2} \dif x = \prod_{i=1}^d \dif \mu(x_i).
\end{equation}
The following proposition is a special case of a standard result on exactness of tensor product rules~\citep[Section 2.4]{Oettershagen2017}.

\begin{proposition}\label{prop:tensor} Consider the tensor product rule~\eqref{eq:tensorRule} and suppose that, for each $i=1,\ldots,d$, $Q_i(\phi^i_n) = \mu(\phi^i_n)$ for some functions $\phi^i_1,\ldots,\phi_{N_i}^i \colon \R \to \R$. Then
\begin{equation*}
Q^d(f) = \mu^d(f) \hspace{0.5cm} \text{for every} \hspace{0.5cm} f \in \lspan \Set[\big]{ \textstyle \prod_{i=1}^d \phi_{\mathcal{I}(i)}^i }{ \mathcal{I} \leq \mathcal{N} }.
\end{equation*}
\end{proposition}

When a multivariate kernel is \emph{separable}, this result can be used in constructing kernel cubature rules out of kernel quadrature rules. We consider $d$-dimensional separable Gaussian kernels
\begin{equation}\label{eq:multiDKernel}
k^d(x,y) \coloneqq \exp\bigg( \! - \frac{1}{2} \sum_{i=1}^d \frac{(x_i - y_i)^2}{\ell_i^2} \bigg) = \prod_{i=1}^d \exp\bigg( \! -\frac{(x_i-y_i)^2}{2\ell_i^2} \bigg) \eqqcolon \prod_{i=1}^d k_i(x_i,y_i),
\end{equation}
where $\ell_i$ are dimension-wise length-scales. For each $i=1,\ldots,d$, the kernel quadrature rule $Q_{k,i}$ with nodes $X_i = \{ x_{i,1},\ldots,x_{i,N_i} \} $ and weights $w_{k,1}^i,\ldots,w_{k,N_i}^i$ is, by definition, exact for the $N_i$ kernel translates at the nodes:
\begin{equation*}
Q_{k,i}\big(k(x_{i,n}, \cdot)\big) = \mu\big(k(x_{i,n}, \cdot)\big)
\end{equation*}
for each $n=1,\ldots,N_i$. \Cref{prop:tensor} implies that the $d$-dimensional kernel cubature rule $Q_k^d$ at the nodes $X = X_1 \times \cdots \times X_d$ is a tensor product of the univariate rules:
\begin{equation}\label{eq:tensorKQ}
Q_k^d(f) = (Q_{k,1} \otimes \cdots \otimes Q_{k,d})(f) \eqqcolon \sum_{\mathcal{I} \leq \mathcal{N}} w_{k,\mathcal{I}} f(x_\mathcal{I}),
\end{equation}
with the weights being products of univariate Gaussian kernel quadrature weights, \sloppy{${w_{k,\mathcal{I}} = \prod_{i=1}^d w_{k,\mathcal{I}(i)}}$}.
This is the case because each kernel translate $k^d(x,\cdot)$, $x \in X$, can be written as
\begin{equation*}
k^d(x, \cdot) = \prod_{i=1}^d k_i(x_i, \cdot)
\end{equation*}
by separability of $k^d$.

We can extend \Cref{thm:main} to higher dimensions if the node set is a Cartesian product of a number of scaled Gauss--Hermite node sets. For this purpose, for each $i=1,\ldots,d$ we use the $L(\mu_{\alpha_i})^2$-orthonormal eigendecomposition of the Gaussian kernel $k_i$. The eigenfunctions, eigenvalues, and other related constants from \Cref{sec:eigendecomposition} for the eigendecomposition of the $i$th kernel are assigned an analogous subscript. Furthermore, use the notation
\begin{equation*}
\lambda_\mathcal{I} \coloneqq \prod_{i=1}^d \lambda_{\mathcal{I}(i)}^{\alpha_i} \hspace{0.5cm} \text{and} \hspace{0.5cm} \phi_\mathcal{I}(x) = \prod_{i=1}^d \phi_{\mathcal{I}(i)}^{\alpha_i}(x_i).
\end{equation*}

\begin{theorem}\label{thm:mainMultiD} For $i=1,\ldots,d$, let $x_{i,1}^\textsm{GH},\ldots,x_{i,N_i}^\textsm{GH}$ and $w_{i,1}^\textsm{GH}, \ldots w_{i,N_i}^\textsm{GH}$ stand for the nodes and weights of the $N_i$-point Gauss--Hermite quadrature rule and define the nodes
\begin{equation}\label{eq:multiDNdoes}
\xkgh_{i,n} \coloneqq \frac{1}{\sqrt{2} \alpha_i \beta_i} x_{i,n}^\textsm{GH}.
\end{equation}
Then the weights of the tensor product quadrature rule
\begin{equation*}
\widetilde{Q}_k^d(f) \coloneqq \sum_{\mathcal{I} \leq \mathcal{N}} \wkgh_{k,\mathcal{I}} f(\xkgh_{\mathcal{I}}),
\end{equation*}
that is defined by the exactness conditions $\widetilde{Q}_k^d(\phi_\mathcal{I}) = \mu^d(\phi_\mathcal{I})$ for every $\mathcal{I} \leq \mathcal{N}$, are \sloppy{${\widetilde{w}_{k,\mathcal{I}} = \prod_{i=1}^d \widetilde{w}_{k,\mathcal{I}(i)}^i}$} for
\begin{equation*}
\widetilde{w}_{k,n}^i = \bigg(\frac{1}{1+2\delta_i^2}\bigg)^{1/2} w_{i,n}^\textsm{GH} \neper^{\delta^2 \xkgh_{i,n}^2} \sum_{m=0}^{\floor{(N-1)/2}} \frac{1}{2^m m! } \bigg( \! \frac{2\alpha_i^2 \beta_i^2}{1+2\delta_i^2} - 1 \bigg)^m \hermite_{2m}(x_{i,n}^\textsm{GH}),
\end{equation*}
where $\alpha$, $\beta$, and $\delta$ are defined in \Cref{eq:constants} and $\hermite_{2m}$ are the probabilists' Hermite polynomials~\eqref{eq:hermite}. 
\end{theorem}

As in one dimension, the weights $\wkgh_{k,\mathcal{I}}$ are supposed to approximate $w_{k,\mathcal{I}}$. Moreover, convergence rates can be obtained: a tensor product analogues of~\Cref{thm:convergence,thm:ConvergenceSpecific} follow from noting that every function $f \colon \R^d \to \R$ in the RKHS $\rkhs^d$ of $k^d$ admits the multivariate Mercer expansion
\begin{equation*}
f(x) = \sum_{ \mathcal{I} \geq 0 } \lambda_\mathcal{I} \inprod{f}{\phi_\mathcal{I}}_{\rkhs^d} \phi_\mathcal{I}(x).
\end{equation*}
See~\citep{KuoSloanWozniakowski2017} for similar convergence analysis of tensor product Gauss--Hermite rules in~$\rkhs^d$.

\begin{theorem} Let $\alpha_1 = \cdots = \alpha_d = 1/\sqrt{2}$. Suppose that the nodes $x_{i,1},\ldots,x_{i,N_i}$ and weights $w_{1}^i,\ldots,w_{N_i}^i$ of the $N_i$-point quadrature rules $Q_{1,N_1},\ldots,Q_{d,N_d}$ satisfy
\begin{enumerate}
\item $\sup_{1\leq i \leq d} \sum_{n=1}^{N_i} \abs[0]{w_{n}^i} \leq W_\mathcal{N}$ for some $W_\mathcal{N} \geq 1$;
\item $Q_{i,N_i}(\phi_n^\alpha) = \mu(\phi_n^\alpha)$ for each $n = 0,\ldots,M_{N_i}-1$ and $i=1,\ldots,d$ for some $M_{N_i} \geq 1$;
\item $\sup_{1\leq n \leq {N_i}}\abs[0]{x_{i,n}} \leq 2 \sqrt{M_{N_i}} / \beta$ for each $i=1,\ldots,d$.
\end{enumerate}
Define the tensor product rule
\begin{equation*}
Q_\mathcal{N}^d = Q_{1,N_1} \otimes \cdots \otimes Q_{d,N_d}.
\end{equation*}
Then there exist constants $C > 0$, independent of $\mathcal{N}$ and $Q_\mathcal{N}^d$, and $0 < \eta < 1$ such that
\begin{equation*}
e(Q_\mathcal{N}^d) \leq C W_\mathcal{N}^d \eta^{M},
\end{equation*}
where $M = \min(M_{N_1},\ldots,M_{N_d})$. Explicit forms of $C$ and $\eta$ appear in \Cref{eq:ConvConstantsMulti}.
\end{theorem}

\begin{proof} 
The proof is largely analogous to that of \Cref{thm:convergence}. Since $f \in \rkhs^d$ can be written as
\begin{equation*}
f = \sum_{\mathcal{I} \geq 0} \lambda_\mathcal{I} \inprod{f}{\phi_\mathcal{I}}_{\rkhs^d} \phi_{\mathcal{I}},
\end{equation*}
by defining the index set
\begin{equation*}
\mathcal{A}_\mathcal{M} \coloneqq \Set[\big]{\mathcal{I} \in \N^d}{\mathcal{I}(i) \geq M_{N_i} \text{ for at least one $i \in \{1,\ldots,d\}$} } \subset \N^d
\end{equation*}
we obtain
\begin{equation*}
\abs[0]{\mu^d(f) - Q_\mathcal{N}^d(f)} = \abs[3]{\sum_{\mathcal{I} \in \mathcal{A}_{\mathcal{M}}} \lambda_\mathcal{I} \inprod{f}{\phi_\mathcal{I}}_{\rkhs^d} \big[ \mu^d(\phi_{\mathcal{I}}) - Q_\mathcal{N}^d(\phi_{\mathcal{I}}) \big] }.
\end{equation*}
Consequently, the Cauchy--Schwarz inequality yields
\begin{equation}\label{eq:MultiFBound}
\begin{split}
\abs[0]{\mu^d(f) - Q_\mathcal{N}^d(f)} &\leq \norm[0]{f}_{\rkhs^d} \sum_{\mathcal{I} \in \mathcal{A}_{\mathcal{M}}} \lambda_\mathcal{I}^{1/2} \abs[1]{ \mu^d(\phi_{\mathcal{I}}) - Q_\mathcal{N}^d(\phi_{\mathcal{I}}) } \\
&= \norm[0]{f}_{\rkhs^d} \tau^{d/2} \sum_{\mathcal{I} \in \mathcal{A}_{\mathcal{M}}} \lambda^{\abs[0]{\mathcal{I}}/2} \abs[1]{ \mu^d(\phi_{\mathcal{I}}) - Q_\mathcal{N}^d(\phi_{\mathcal{I}}) },
\end{split}
\end{equation}
where we again use the notation
\begin{equation*}
\tau = \sqrt{ \frac{1/2}{1/2+\delta^2 + \epsilon^2} } \hspace{1cm} \text{ and } \hspace{1cm} \lambda = \frac{\epsilon^2}{1/2 + \delta^2 + \epsilon^2}.
\end{equation*}

Since $\mu(\phi_n) \leq 1$ for any $n \geq 0$, integration error for the eigenfunction $\phi_\mathcal{I}$ satisfies
\begin{equation}\label{eq:MultiD-rec1}
\begin{split}
\abs[1]{ \mu^d(\phi_{\mathcal{I}}) - Q_\mathcal{N}^d(\phi_{\mathcal{I}}) } \hspace{-2cm}& \\
={}& \abs[3]{ \prod_{i=1}^d \mu( \phi_{\mathcal{I}(i)} ) - \prod_{i=1}^d Q_{i,N_i}( \phi_{\mathcal{I}(i)} ) } \\
={}& \Bigg\lvert \big[ \mu(\phi_{\mathcal{I}(d)}) - Q_{d,N_d}(\phi_{\mathcal{I}(d)})  \big] \prod_{i=1}^{d-1} \mu( \phi_{\mathcal{I}(i)} ) \bigg. \\
& \bigg. + Q_{d,N_d}(\phi_{\mathcal{I}(d)}) \bigg( \prod_{i=1}^{d-1} \mu( \phi_{\mathcal{I}(i)} ) - \prod_{i=1}^{d-1} Q_{i,N_i}( \phi_{\mathcal{I}(i)} ) \bigg) \Bigg\rvert \\
\leq{}& \abs[1]{\mu(\phi_{\mathcal{I}(d)}) - Q_{d,N_d}(\phi_{\mathcal{I}(d)})} \\
&+ \abs[1]{Q_{d,N_d}(\phi_{\mathcal{I}(d)})} \abs[3]{ \prod_{i=1}^{d-1} \mu( \phi_{\mathcal{I}(i)} ) - \prod_{i=1}^{d-1} Q_{i,N_i}( \phi_{\mathcal{I}(i)} ) }.
\end{split}
\end{equation}
\rev{
Define the index sets $\mathcal{B}_\mathcal{M}^j(\mathcal{I}) = \Set{ j \leq i \leq d }{ \mathcal{I}(i) \geq M_{N_i}}$ and their cardinalities $b_\mathcal{M}^j(\mathcal{I}) = \#\mathcal{B}_\mathcal{M}^j(\mathcal{I}) \leq d-j+1$ for $j \geq 1$.
Because $\abs[0]{\mu(\phi_{\mathcal{I}(i)}) - Q_{i,N_i}(\phi_{\mathcal{I}(i)})} = 0$ and $\abs[0]{Q_{i,N_{i}}(\phi_{\mathcal{I}(i)})} = \abs[0]{\mu(\phi_{\mathcal{I}(i)})} \leq 1$ if $\mathcal{I}(i) < M_{N_i}$, expansion of the recursive inequality~\eqref{eq:MultiD-rec1} gives
\begin{equation}\label{eq:Multi-rec-expanded}
\begin{split}
\abs[1]{ \mu^d(\phi_{\mathcal{I}}) - Q_\mathcal{N}^d(\phi_{\mathcal{I}}) } \hspace{-2cm}& \\
&\leq \sum_{i=1}^d \abs[1]{\mu(\phi_{\mathcal{I}(i)}) - Q_{i,N_i}(\phi_{\mathcal{I}(i)})} \prod_{j=i+1}^{d} \abs[1]{Q_{j,N_{j}}(\phi_{\mathcal{I}(j)})} \\
&= \sum_{i \in \mathcal{B}_\mathcal{M}^1(\mathcal{I})} \abs[1]{\mu(\phi_{\mathcal{I}(i)}) - Q_{i,N_i}(\phi_{\mathcal{I}(i)})} \prod_{j=i+1}^{d} \abs[1]{Q_{j,N_{j}}(\phi_{\mathcal{I}(j)})} \\
&\leq \sum_{i \in \mathcal{B}_\mathcal{M}^1(\mathcal{I})} \abs[1]{\mu(\phi_{\mathcal{I}(i)}) - Q_{i,N_i}(\phi_{\mathcal{I}(i)})} \prod_{j \in \mathcal{B}_\mathcal{M}^{i+1}(\mathcal{I})} \abs[1]{Q_{j,N_{j}}(\phi_{\mathcal{I}(j)})}.
\end{split}
\end{equation}
\Cref{eq:PhiError} provides the bounds \sloppy{${\abs[0]{\mu(\phi_{\mathcal{I}(i)}) - Q_{i,N_i}(\phi_{\mathcal{I}(i)})} \leq 1 + K \sqrt{\beta} W_{\mathcal{N}} \neper^{M_{N_i}/\beta^2}}$} and $\abs[0]{Q_{i,N_i}(\phi_{\mathcal{I}(i)})} \leq K \sqrt{\beta} W_\mathcal{N} \neper^{M_{N_i}/\beta^2}$ for the constant $K = 1.087$ that, when plugged in \Cref{eq:Multi-rec-expanded}, yield
\begin{equation}\label{eq:MultiPhiIBound}
\begin{split}
\abs[1]{ \mu^d(\phi_{\mathcal{I}}) - Q_\mathcal{N}^d(\phi_{\mathcal{I}}) } \hspace{-2.8cm}& \\
&\leq \sum_{i \in \mathcal{B}_\mathcal{M}^1(\mathcal{I})} \big( 1 + K \sqrt{\beta} W_{\mathcal{N}} \neper^{M_{N_i}/\beta^2} \big) \prod_{j \in \mathcal{B}_\mathcal{M}^{i+1}(\mathcal{I})} K \sqrt{\beta} W_\mathcal{N} \neper^{M_{N_j}/\beta^2} \\
&= \sum_{i \in \mathcal{B}_\mathcal{M}^1(\mathcal{I})} \big( 1 + K \sqrt{\beta} W_{\mathcal{N}} \neper^{M_{N_i}/\beta^2} \big) \big(K \sqrt{\beta} W_\mathcal{N} \big)^{b_\mathcal{M}^{i+1}(\mathcal{I})} \exp\Bigg( \frac{1}{\beta^2} \sum_{ j \in \mathcal{B}_\mathcal{M}^{i+1}(\mathcal{I}) } M_{N_j} \Bigg) \\
&\leq 2 \sum_{i \in \mathcal{B}_\mathcal{M}^1(\mathcal{I})} \big(K \sqrt{\beta} W_\mathcal{N} \big)^{b_\mathcal{M}^{i}(\mathcal{I})} \exp\Bigg( \frac{1}{\beta^2} \sum_{ j \in \mathcal{B}_\mathcal{M}^{i}(\mathcal{I}) } M_{N_j} \Bigg),
\end{split}
\end{equation}
where the last inequality is based on the facts that $i \in \mathcal{B}_\mathcal{M}^i(\mathcal{I})$ if \sloppy{${i \in \mathcal{B}_\mathcal{M}^1(\mathcal{I})}$} and \sloppy{${1 + K \sqrt{\beta} W_{\mathcal{N}} \neper^{M_{N_i}/\beta^2} \leq 2 K \sqrt{\beta} W_{\mathcal{N}} \neper^{M_{N_i}/\beta^2}}$}, a consequence of \sloppy{${K, \beta, W_\mathcal{N} \geq 1}$}.
\Cref{eq:MultiPhiIBound,eq:MultiFBound}, together with \Cref{lemma:ConvConstant}, now yield
\begin{align*}
\lvert &\mu^d(f) - Q_\mathcal{N}^d(f) \rvert \\
&\leq 2 \norm[0]{f}_{\rkhs^d} \tau^{d/2} \sum_{\mathcal{I} \in \mathcal{A}_{\mathcal{M}}} \lambda^{\abs[0]{\mathcal{I}}/2} \sum_{ i \in \mathcal{B}_\mathcal{M}^1(\mathcal{I})} \big(K \sqrt{\beta} W_\mathcal{N} \big)^{b_\mathcal{M}^i(\mathcal{I})} \exp\Bigg( \frac{1}{\beta^2} \sum_{ j \in \mathcal{B}_\mathcal{M}^i(\mathcal{I}) } M_{N_j} \Bigg) \\
&\leq 2 \norm[0]{f}_{\rkhs^d} \tau^{d/2} \big(K \sqrt{\beta} W_\mathcal{N} \big)^d \sum_{\mathcal{I} \in \mathcal{A}_{\mathcal{M}}} \lambda^{\abs[0]{\mathcal{I}}/2} \sum_{ i \in \mathcal{B}_\mathcal{M}^1(\mathcal{I})} \exp\Bigg( \frac{1}{\beta^2} \sum_{ j \in \mathcal{B}_\mathcal{M}^i(\mathcal{I}) } M_{N_j} \Bigg) \\
&\leq 2d \norm[0]{f}_{\rkhs^d} \tau^{d/2} \big(K \sqrt{\beta} W_\mathcal{N} \big)^d \sum_{\mathcal{I} \in \mathcal{A}_{\mathcal{M}}} \lambda^{\abs[0]{\mathcal{I}}/2} \neper^{\abs[0]{\mathcal{I}} / \beta^2} \\
&= 2d \norm[0]{f}_{\rkhs^d} \big(K \sqrt{\tau \beta} W_\mathcal{N} \big)^d \sum_{\mathcal{I} \in \mathcal{A}_{\mathcal{M}}} \big(\sqrt{\lambda} \neper^{1/\beta^2}\big)^{\abs[0]{\mathcal{I}}} \\
&\leq 2d \norm[0]{f}_{\rkhs^d} \big(K \sqrt{\tau \beta} W_\mathcal{N} \big)^d \big(\sqrt{\lambda} \neper^{1/\beta^2}\big)^M \sum_{\mathcal{I} \geq 0} \big(\sqrt{\lambda} \neper^{1/\beta^2}\big)^{\abs[0]{\mathcal{I}}} \\
&= 2d \norm[0]{f}_{\rkhs^d} \big(K \sqrt{\tau \beta} W_\mathcal{N} \big)^d \big(\sqrt{\lambda} \neper^{1/\beta^2}\big)^M \bigg( \frac{1}{1-\sqrt{\lambda} \neper^{1/\beta^2}} \bigg)^d.
\end{align*}
The claim therefore holds with
\begin{equation} \label{eq:ConvConstantsMulti}
C = 2d \bigg( \frac{K \sqrt{\tau \beta} }{1-\sqrt{\lambda} \neper^{1/\beta^2}} \bigg)^d \quad \text{ and } \quad \eta = \sqrt{\lambda} \neper^{1/\beta^2} < 1.
\end{equation}
}
\qed
\end{proof}

A multivariate version of \Cref{thm:ConvergenceSpecific} is obvious.

\section{Numerical experiments}\label{sec:numsim}

This section contains numerical experiments on properties and accuracy of the approximate Gaussian kernel quadrature weights defined in \Cref{thm:main,thm:mainMultiD}. The experiments have been implemented in MATLAB, and they are available at \texttt{https://github.com/tskarvone/gauss-mercer}. The value $\alpha = 1/\sqrt{2}$ is used in all experiments. The experiments indicate that
\begin{enumerate}
\item Computation of the approximate weights in \Cref{eq:approximation} is numerically stable.
\item The weight approximation is quite accurate, its accuracy increasing with the number of nodes and the length-scale, as predicted in \Cref{sec:lengthscale}.
\item The weights $w_{k,n}$ and $\widetilde{w}_{k,n}$ are positive for every $N$ and $n=1,\ldots,N$ and their sums converge to one exponentially in $N$.
\item The quadrature rule $\widetilde{Q}_{k}$ converges exponentially, as implied by \Cref{thm:ConvergenceSpecific} and empirical observations on the behaviour of its weights.
\item In numerical integration of specific functions, the approximate kernel quadrature rule $\widetilde{Q}_{k}$ can achieve integration accuracy almost indistinguishable from that of the corresponding Gaussian kernel quadrature rule $Q_{k}$ and superior to some more traditional alternatives.
\end{enumerate}
This suggest~\Cref{eq:approximation} can be used as an accurate and numerically stable surrogate for computing the Gaussian kernel quadrature weights when the naive approach based on solving the linear system~\eqref{eq:linSys} is precluded by ill-conditioning of the kernel matrix. Furthermore, the choice~\eqref{eq:nodes} of the nodes by scaling the Gauss--Hermite nodes appears to yield an exponentially convergent kernel quadrature rule that has positive weights.

\begin{figure}[t!]
  \centering
  \includegraphics[width=\textwidth]{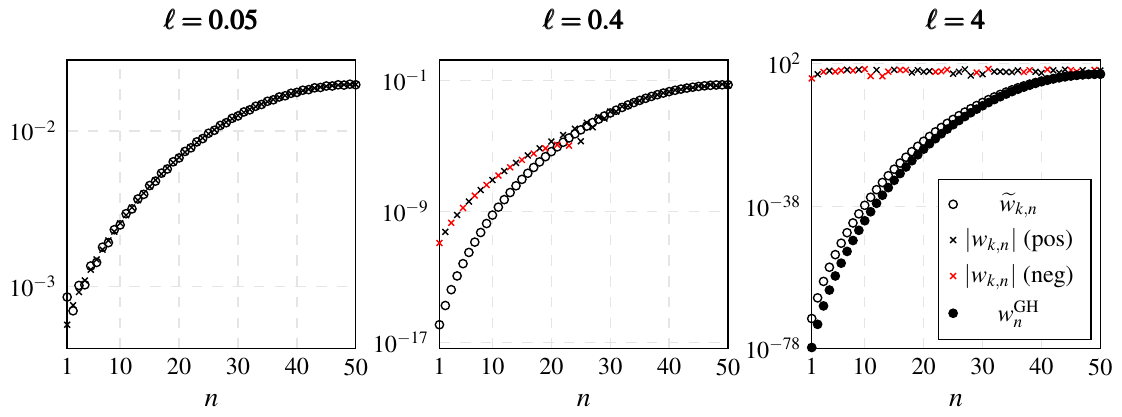}
  \caption{\rev{Absolute kernel quadrature weights, as computed directly from the linear system~\eqref{eq:linSys}, and the approximate weights~\eqref{eq:approximation} for $N = 99$, nodes $\tilde{x}_{k,n}$, and three different length-scales. Red is used to indicate those of $w_{k,n}$ that are negative. The nodes are in ascending order, so by symmetry it is sufficient to display weights only for $n=1,\ldots,50$ (in fact, $w_{k,n}$ are not necessarily numerically symmetric; see \Cref{sec:numsim-approximation}). The Gauss--Hermite nodes and weights were computed using the Golub--Welsch algorithm~\citep[Section~3.1.1.1]{Gautschi2004} and MATLAB's variable precision arithmetic. \Cref{eq:approximation} did not present any numerical issues as the sum, which can contain both positive and negative terms, was always dominated by the positive terms and all its terms were of reasonable magnitude.}}\label{fig:stability}
\end{figure}

\subsection{Numerical stability and distribution of weights}\label{sec:numsim-stability}

\rev{
We have not encountered any numerical issues when computing the approximate weights~\eqref{eq:approximation}. In this example we set $N=99$ and examine the distribution of approximate weights $\widetilde{w}_{k,n}$ for $\ell = 0.05$, $\ell = 0.4$ and $\ell = 4$. \Cref{fig:stability} depicts (i) approximate weights $\widetilde{w}_{k,n}$, (ii) absolute kernel quadrature weights $\abs[0]{w_{k,n}}$  obtained by solving the linear system~\eqref{eq:linSys} for the points $\tilde{x}_{n}$ and, for $\ell = 4$, (iii) Gauss--Hermite weights $w_n^\textsm{GH}$. The approximate weights $\widetilde{w}_{k,n}$ display no signs of numerical instabilities; their magnitudes vary smoothly and all of them are positive. That $\widetilde{w}_{k,1} > \widetilde{w}_{k,2}$ for $\ell = 0.05$ appears to be caused by the sum in \Cref{eq:approximation} having not converged yet: the constant $2\alpha^2\beta^2/(1+2\delta^2) - 1$, that controls the rate of convergence of this sum, converges to $1$ as $\ell \to 0$ (in this case its value is $0.9512$) and $H_{2m}(x_1^\textsm{GH}) > 0$ for every $m = 1,\ldots,49$ while $H_{2m}(x_n^\textsm{GH}) < 0$ for $m = 46,47,48,49$. This and further experiments in \Cref{sec:numsim-approximation} merely illustrates that quality of the weight approximation deteriorates when $\ell$ is small---as predicted in \Cref{sec:lengthscale}.
Behaviour of $\widetilde{w}_{k,n}$ is in stark contrast to the naively computed weights $w_{k,n}$ that display clear signs of numerical instabilities for $\ell = 0.4$ and $\ell = 4$ (condition numbers of the kernel matrices were roughly $2.66 \times 10^{16}$ and $3.59 \times 10^{18}$).
Finally, the case $\ell = 4$ provides further evidence for numerical stability of \Cref{eq:approximation} since, based on \Cref{sec:lengthscale}, $\widetilde{w}_{k,n} \to w_n^\textsm{GH}$ as $\ell \to \infty$ and, furthermore, there is reason to believe that $w_{k,n}$ would share this property if they were computed in arbitrary-precision arithmetic. \Cref{sec:numsim-positivity} and the experiments reported by \citet{FasshauerMcCourt2012} provide additional evidence for numerical stability of \Cref{eq:approximation}.
}

\subsection{Accuracy of the weight approximation}\label{sec:numsim-approximation}

\begin{figure}[t!]
  \centering
  \includegraphics[width=\textwidth]{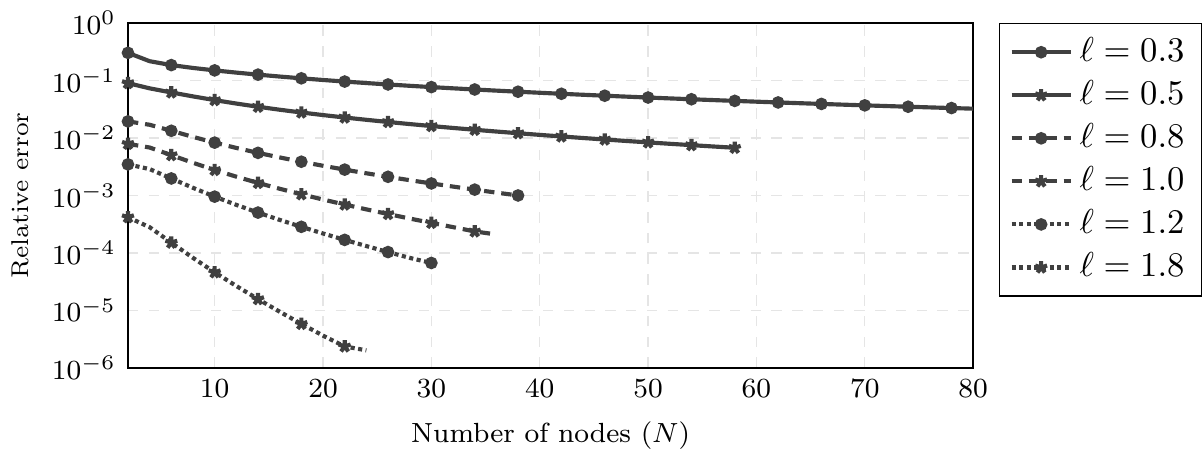}
  \caption{Relative weight approximation error~\eqref{eq:RelErr} for different length-scales.}\label{fig:accuracy}
\end{figure}

Next we assess quality of the weight approximation $\widetilde{w}_{k} \approx w_{k}$. \Cref{fig:accuracy} depicts the results for a number of different length-scales in terms of norm of the relative weight error,
\begin{equation}\label{eq:RelErr}
\sqrt{ \sum_{n=1}^N \bigg( \frac{w_{k,n} - \widetilde{w}_{k,n}}{w_{k,n}} \bigg)^2 }.
\end{equation}
As the kernel matrix quickly becomes ill-conditioned, computation of the kernel quadrature weights $w_{k}$ is challenging, particularly when the length-scale is large. To partially mitigate the problem we replaced the kernel quadrature weights with their QR decomposition approximations $\widetilde{w}_k^M$ derived in \Cref{sec:qr}. The truncation length~$M$ was selected based on machine precision; see \citep[Section 4.2.2]{FasshauerMcCourt2012} for details. Yet even this does not work for large enough $N$. Because kernel quadrature rules on symmetric point sets have symmetric weights~\citep[Section 5.2.4]{KarvonenSarkka2018,Oettershagen2017}, breakdown in symmetricity of the computed kernel quadrature weights was used as a heuristic proxy for emergence of numerical instability: for each length-scale, relative errors are presented in \Cref{fig:accuracy} until the first $N$ such that $\abs[0]{1-w_{k,N}/w_{k,1}} > 10^{-6}$, ordering of the nodes being from smallest to the largest so that $w_{k,N} = w_{k,1}$ in absence of numerical errors.

\subsection{Properties of the weights}\label{sec:numsim-positivity}

\Cref{fig:weights} shows the minimal weights $\min_{n=1,\ldots,N} \widetilde{w}_{k,n}$ and convergence to one of $\sum_{n=1}^N \abs[0]{\widetilde{w}_{k,n}}$ for a number of different length-scales. These results provide strong numerical evidence for the conjecture that $\widetilde{w}_{k,n}$ remain positive and that the assumptions of \Cref{thm:ConvergenceSpecific} hold. Exact weights, as long as they can be reliably computed (see \Cref{sec:numsim-approximation}), exhibit behaviour practically indistinguishable from the approximate ones and are not therefore depicted separately in \Cref{fig:weights}.

\begin{figure}[t!]
  \centering
  \includegraphics[width=\textwidth]{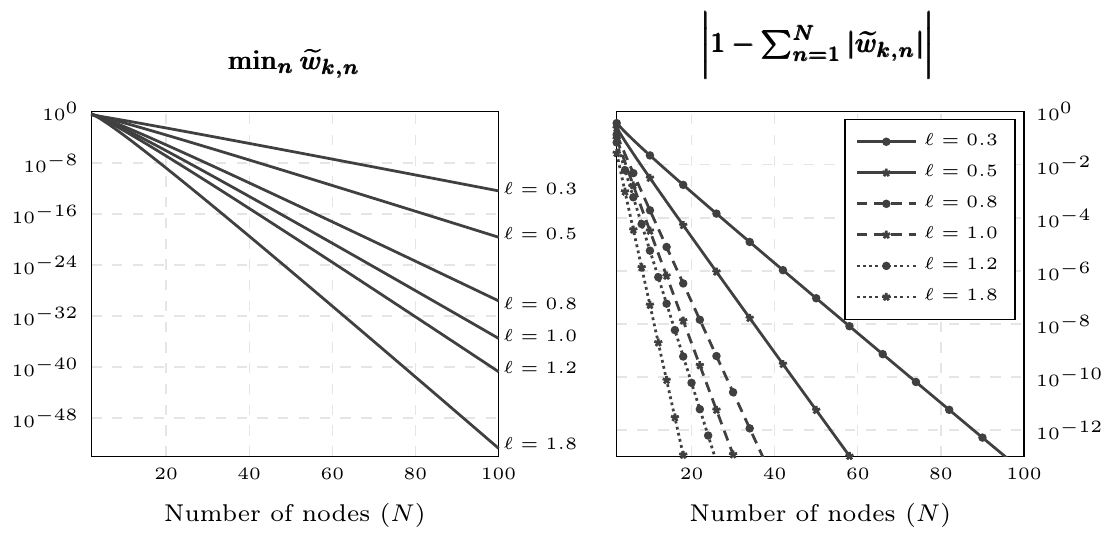}
  \caption{Minimal weights and convergence to one of the the sum of absolute values of the weights for six different length-scales.}\label{fig:weights}
\end{figure}

\subsection{Worst-case error}\label{sec:experiment-wce}

The worst-case error $e(Q)$ of a quadrature rule $Q(f) = \sum_{n=1}^N w_n f(x_n)$ in a reproducing kernel Hilbert space induced by the kernel $k$ is explicitly computable:
\begin{equation}\label{eq:WCE-explicit}
e(Q)^2 = \mu(k_\mu) + \sum_{n,m = 1}^N w_n w_m k(x_n,x_m) - 2 \sum_{n=1}^N w_n k_\mu(x_n).
\end{equation}
\Cref{fig:wce} compares the worst-case errors in the RKHS of the Gaussian kernel for six different length-scales of (i) the classical Gauss--Hermite quadrature rule, (ii) the quadrature $\widetilde{Q}_{k}(f) = \sum_{n=1}^N \widetilde{w}_{k,n} f(\tilde{x}_n)$ of \Cref{thm:main}, and (iii) the kernel quadrature rule with its nodes placed uniformly between the largest and smallest of $\tilde{x}_n$. We observe that $\widetilde{Q}_{k}$ is, for all length-scales, the fastest of these rules to converge (the kernel quadrature rule at $\tilde{x}_n$ yields WCEs practically indistinguishable from those of $\widetilde{Q}_{k}$ and is therefore not included). It also becomes apparent that the convergence rates derived in \Cref{thm:convergence,thm:ConvergenceSpecific} for $\widetilde{Q}_{k}$ are rather conservative. For example, for $\ell = 0.2$ and $\ell = 1$ the empirical rates are $e(\widetilde{Q}_{k}) = \mathcal{O}(\neper^{-c N})$ with $c \approx 0.21$ and $c \approx 0.98$, respectively, whereas \Cref{eq:ConvConstants} yields the theoretical values $c \approx 0.00033$ and $c \approx 0.054$, respectively.

\begin{figure}[t!]
  \centering
  \includegraphics[width=\textwidth]{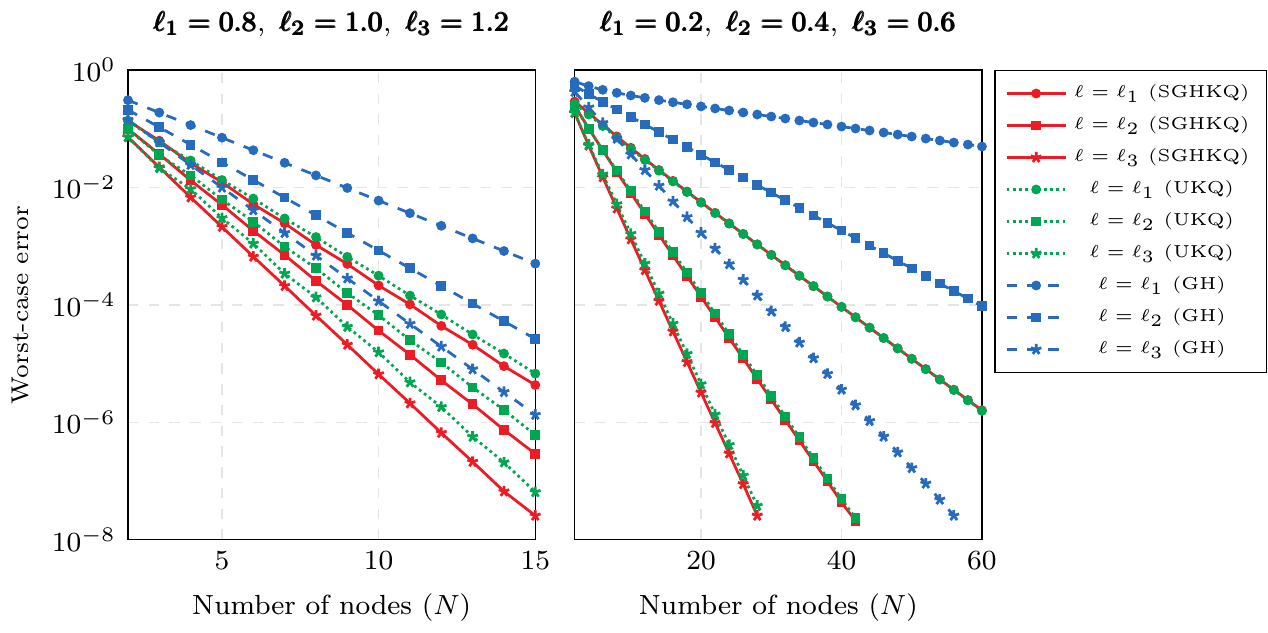}
  \caption{Worst-case errors~\eqref{eq:WCE-explicit} in the Gaussian RKHS as functions of the number of nodes of the quadrature rule of \Cref{thm:main} (SGHKQ), the kernel quadrature rule with nodes placed uniformly between the largest and smallest of $\tilde{x}_n$ (UKQ), and the Gauss--Hermite rule (GH). WCEs are displayed until the square root of floating-point relative accuracy ($\approx 1.4901 \times 10^{-8}$) is reached.}\label{fig:wce}
\end{figure}

\subsection{Numerical integration}

Set $\ell = 1.2$ and consider the integrand
\begin{equation}\label{eq:test-function}
f(x) = \prod_{i=1}^d \exp\bigg( \! -\frac{c_i x^2}{2\ell^2} \bigg) x^{m_i}.
\end{equation}
When $0 < c_i < 4$ and $m_i \in \N$ for each $i=1,\ldots,d$, the function is in $\rkhs$~\citep[Theorems 1 and 3]{Minh2010}. Furthermore, the Gaussian integral of this function is available in closed form:
\begin{equation*}
(2\pi)^{-d/2} \int_{\R^d} f(x) \neper^{-\norm[0]{x}^2/2} \dif x = \prod_{i=1}^d \frac{m_i!}{2^{m_i/2} (m_i/2)!} \bigg( \frac{\ell}{\sqrt{c_i}} \bigg)^{m_i+1} \bigg( \frac{1}{1+\ell^2/c_i}\bigg)^{(m_i+1)/2}
\end{equation*}
when $m_i$ are even (when they are not even, the integral is obviously zero). \Cref{fig:int} shows integration error of the three methods (or, in higher dimensions, their tensor product versions) used in \Cref{sec:experiment-wce} and the kernel quadrature rule based on the nodes $\tilde{x}_n$ for (i) $d=1$, $m_1 = 6$, $c_1 = 3/2$ and (ii) $d=3$, $m_1=6$, $m_2 = 4$, $m_3 = 2$, $c_1 = 3/2$, $c_2 = 3$, $c_3 = 1/2$. As expected, there is little difference between $\widetilde{Q}_{k}$ and $Q_{k}$.

\begin{figure}[t!]
  \centering
  \includegraphics[width=\textwidth]{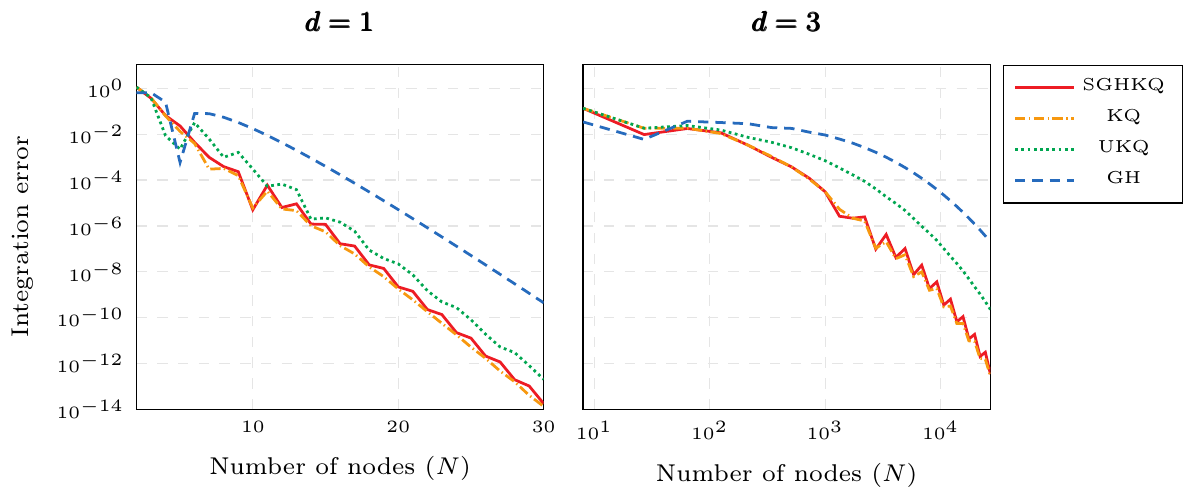}
  \caption{Error in computing the Gaussian integral of the function~\eqref{eq:test-function} in dimensions one and three using the quadrature rule of \Cref{thm:main} (SGHKQ), the corresponding kernel quadrature rule (KQ), the kernel quadrature rule with nodes placed uniformly between the largest and smallest of $\tilde{x}_n$ (UKQ), and the Gauss--Hermite rule (GH). Tensor product versions of these rules are used in dimension three.}\label{fig:int}
\end{figure}


\bibliographystyle{apa}

\providecommand{\BIBYu}{Yu}

\end{document}